\documentclass[a4paper,UKenglish,numberwithinsect]{lipics_md}
\usepackage{graphicx}
\usepackage{color}
\usepackage{amsmath}
\usepackage{amssymb}
\usepackage{xspace}
\usepackage{epsfig}
\usepackage[dvipsnames]{xcolor}
\usepackage{soul}
\usepackage{xfrac}

\usepackage{enumerate}

\newcommand {\mm}[1] {\ifmmode{#1}\else{\mbox{\(#1\)}}\fi}

\newcommand {\scalprod}[2] {{\langle #1 , #2 \rangle}}
\newcommand{\denselist}{\itemsep 0pt\parsep=1pt\partopsep 0pt}

\newcommand{\ignore}[1]{}

\newsavebox{\smallProofsym}                 
\savebox{\smallProofsym}                               %
{
\begin{picture}(6,6)
\put(0,0){\framebox(6,6){}}
\put(0,2){\framebox(4,4){}}
\end{picture} 
}  

\makeatletter
\long\def\@makecaption#1#2{%
  \vskip\abovecaptionskip
  \sbox\@tempboxa{\small #1: #2}%
  \ifdim \wd\@tempboxa >\hsize
    \small #1: #2\par
  \else
    \global \@minipagefalse
    \hb@xt@\hsize{\hfil\box\@tempboxa\hfil}%
  \fi
  \vskip\belowcaptionskip}
\makeatother

\theoremstyle{plain}

\newtheorem*{supermaintheorem*}{Main Theorem}

\newtheorem*{supermaincorollary*}{Main Corollary}

\newcommand{\Rspace}        {\mm{{\mathbb R}}}

\newcommand{\Zspace}        {\mm{{\mathbb Z}}}

\newcommand{\Voronoi}[2]    {\mm{{\rm Vor}_{#1}{({#2})}}}
\newcommand{\wVoronoi}[3]   {\mm{{\rm Vor}_{#1}{({#2},{#3})}}}
\newcommand{\Delaunay}[2]   {\mm{{\rm Del}_{#1}{({#2})}}}
\newcommand{\wDelaunay}[3]  {\mm{{\rm Del}_{#1}{({#2},{#3})}}}
\newcommand{\Polyhedron}[2] {\mm{{\mathcal P}_{#1}{({#2})}}}
\newcommand{\wPolyhedron}[3]{\mm{{\mathcal P}_{#1}{({#2},{#3})}}}
\newcommand{\Brillouin}[2]  {\mm{{\rm Bri}_{#1}{({#2})}}}
\newcommand{\Iglesias}[2]   {\mm{{\rm Igl}_{#1}{({#2})}}}

\newcommand{\domain}[2]     {\mm{{\rm dom}_{#1}{({#2})}}}
\newcommand{\Angle}[1]      {\mm{{\angle}{#1}}}
\newcommand{\intensity}     {\mm{\rho}}
\newcommand{\In}[1]         {\mm{{\rm In}{({#1})}}}
\newcommand{\On}[1]         {\mm{{\rm On}{({#1})}}}

\newcommand{\card}[1]       {\mm{{\#}{#1}}}

\newcommand{\norm}[1]       {\mm{\|{#1}\|}}
\newcommand{\Edist}[2]      {\mm{\|{#1}-{#2}\|}}

\newcommand{\ee}            {\mm{\varepsilon}}

\newcommand{\aaa}           {\mm{\bf a}}
\newcommand{\aaaaa}         {\mm{\bar{\bf a}}}
\newcommand{\point}[1]      {\mm{{\rm pt}{({#1})}}}
\newcommand{\height}[1]     {\mm{{\rm ht}{({#1})}}}
\newcommand{\pointt}[2]     {\mm{{\rm pt}{({#1},{#2})}}}
\newcommand{\heightt}[2]    {\mm{{\rm ht}{({#1},{#2})}}}

\newcommand{\Skip}[1]       {}

\definecolor{blue-green}{rgb}{0.0, 0.87, 0.87}

\title{On Angles in Higher Order Brillouin Tessellations and Related Tilings in the Plane\footnote{
  Work by all authors but the second is supported by the European Research Council (ERC), grant no.\ 788183, by the Wittgenstein Prize, Austrian Science Fund (FWF), grant no.\ Z 342-N31, and by the DFG Collaborative Research Center TRR 109, Austrian Science Fund (FWF), grant no.\ I 02979-N35.
  Work by the second author is partially supported by the Alexander von Humboldt Foundation.
}
}
\titlerunning{On Angles in Higher Order Brillouin Tessellations and Related Tilings in the Plane}

\author[1]{Herbert Edelsbrunner}
\affil[1]{IST Austria (Institute of Science and Technology Austria), Kloster\-neu\-burg, Austria, \texttt{herbert.edelsbrunner@ist.ac.at}, \texttt{teresa.heiss@ist.ac.at}, \texttt{morteza.saghafian@ist.ac.at}}

\author[2]{Alexey Garber}
\affil[2]{School of Mathematical and Statistical Sciences, University of Texas Rio Grande Valley, Brownsville, Texas, USA, \texttt{alexey.garber@utrgv.edu}}

\author[3]{Mohadese Ghafari}
\affil[3]{Department of Mathematical Sciences, Sharif University of Technology, Tehran, Iran, \texttt{mohadese.ghafari@ist.ac.at}}

\author[1]{Teresa Heiss}

\author[1]{Morteza Saghafian}

\authorrunning{Edelsbrunner, Garber, Ghafari, Heiss, Saghafian}


\begin{document}
\maketitle

\begin{abstract}
  For a locally finite set in $\Rspace^2$, the order-$k$ Brillouin tessellations form an infinite sequence of convex face-to-face tilings of the plane.
  If the set is coarsely dense and generic, then the corresponding infinite sequences of minimum and maximum angles are both monotonic in $k$.
  As an example, a stationary Poisson point process in $\Rspace^2$ is locally finite, coarsely dense, and generic with probability one.
  For such a set, the distribution of angles in the Voronoi tessellations, Delaunay mosaics, and Brillouin tessellations are independent of the order and can be derived from the formula for angles in order-$1$ Delaunay mosaics given by Miles in 1970.
\end{abstract}

\section{Introduction}
\label{sec:1}

The starting point for the work reported in this paper is a computational experiment conducted as part of a general geometric study of Brillouin zones \cite{EGGHSW21}.
Computing the minimum angles in the $k$-th Brillouin zones of a point in a $2$-dimensional lattice, we noticed that these angles vary monotonically with $k$.
The goal of this paper is to shed additional light on this phenomenon: to generalize, to prove, and to relate to prior knowledge.

\medskip
The most famous result on angles in Delaunay mosaics is Sibson's Maxmin Angle Theorem \cite{Sib78}, which asserts that among all triangulations of a (generic) finite set in $\Rspace^2$, the Delaunay mosaic maximizes the vector of angles sorted in the increasing order lexicographically.
The theorem compares the Delaunay mosaic with other ways to connect the points to form a triangulation.
In contrast, our main result compares the Delaunay mosaic of a set with the higher-order Delaunay mosaics of the same set.
To be specific, we write $\Delaunay{k}{A}$ for the order-$k$ Delaunay mosaic of a set $A \subseteq \Rspace^2$, noting that it is dual to the perhaps better known order-$k$ Voronoi tessellation of $A$.
Writing $\alpha (\Delaunay{k}{A})$ for the infimum angle in the order-$k$ Delaunay mosaic, we prove that $\alpha (\Delaunay{k}{A}) \geq \alpha (\Delaunay{k+1}{A})$ for $k \geq 1$.
This inequality holds when $A$ is locally finite, coarsely dense, and generic; see Section~\ref{sec:2} for the definitions.
Importantly, the inequality is not necessarily true if $A$ is finite.

The inequality for the infimum angles generalizes to order-$k$ Brillouin tessellations (introduced as \emph{degree-$k$ Voronoi diagrams} in \cite{EdSe86})  and to order-$k$ Iglesias mosaics (duals of the order-$k$ Brillouin tessellations), but not to order-$k$ Voronoi tessellations.
Most interesting is however that it holds for the order-$k$ Brillouin tessellations even for points in non-generic position, while this is not true for the order-$k$ Delaunay and Iglesias mosaics.

\medskip
Examples of locally finite and coarsely dense sets are lattices as well as Delone sets, which have packing radius bounded away from zero and covering radius bounded away from infinity.
Another example is a stationary Poisson point process, which is also generic with probability one.
The angle distribution of the (order-$1$) Delaunay mosaic of such a process in $\Rspace^2$ has been determined by Miles \cite{Mil70}.
By the independence of the shape and size of triangles in such a process \cite{ENR17}, the angles of order-$k$ Delaunay mosaics follow the same distribution.
Since order-$k$ Voronoi tessellations are orthogonally dual to these mosaics, their angles follow the symmetric distribution.
The sum (or rather average) of the two distributions is concave and governs the angles of the order-$k$ Brillouin tessellations and Iglesias mosaics.

\medskip \noindent \textbf{Outline.}
Section~\ref{sec:2} introduces background on Voronoi tessellations and Delaunay mosaics, which includes weighted and higher-order versions as well as the related Brillouin tessellations and Iglesias mosaics.
Section~\ref{sec:3} studies the angles of these tessellations and mosaics and proves their monotonicity for locally finite and coarsely dense sets in $\Rspace^2$.
Section~\ref{sec:4} considers the special case of stationary Poisson point processes and characterizes the angle distributions of the tessellations and mosaics.
Section~\ref{sec:5} concludes the paper with a short discussion.

\section{Mosaics and Tessellations}
\label{sec:2}

Given a locally finite set in Euclidean space, its Voronoi tessellation and Delaunay mosaic are dual tilings.
Moving the focus from individual points to subsets of fixed size, $k$, we get the order-$k$ Voronoi tessellation and order-$k$ Delaunay mosaic, which are again dual tilings of the space.
Using a duality introduced in \cite{Aur90}, these generalized tilings are, at the same time, Voronoi tessellations and Delaunay mosaics of weighted point sets.
In this section, we discuss these concepts in the planar case.

\subsection{Orthogonal Dual}
\label{sec:2.1}

We consider \emph{convex face-to-face tilings} of the plane, by which we mean countable and locally finite collections of closed convex polygons (\emph{tiles}) that cover $\Rspace^2$ in such a way that any two tiles are either disjoint or overlap in a common edge or vertex. 
More formally, a convex face-to-face tiling is a complex consisting of convex polygons, edges, and vertices, whose underlying space is $\Rspace^2$.
\begin{definition}[Orthogonal Dual]
  \label{def:orthogonal_dual}
  Let $V$ be a convex face-to-face tiling of $\Rspace^2$.
  Another such tiling, $D$, is an \emph{orthogonal dual} of $V$ if there is an incidence-preserving and dimension-reversing bijection $\beta \colon V \to D$ such that
  \medskip \begin{enumerate}[(i)]
    \item $e$ is orthogonal to $\beta (e)$ for every edge $e \in V$,
    \item if $e$ is shared by tiles $t_1$ on the left and $t_2$ on the right of $e$, then $\beta (t_1)$ is the left and $\beta (t_2)$ is the right endpoint of $\beta (e)$.
  \end{enumerate} \medskip
\end{definition}
To unpack this definition, we note that $\beta$ maps tiles to vertices, edges to edges, and vertices to tiles, such that $c \subseteq d$ in $V$ iff $\beta (d) \subseteq \beta (c)$ in $D$.
Condition (i) requires that the lines that contain $e \in V$ and $\beta (e) \in D$ intersect at a right angle.
Orient the line of $e$ arbitrarily and let $t_1$ and $t_2$ be the tiles that share $e$ and lie to the left and the right of this line, respectively.
Then Condition (ii) requires that the line of $\beta (e)$, which we orient from $\beta (t_1)$ to $\beta (t_2)$, crosses the line of $e$ from left to right.

\medskip
Note that being an orthogonal dual is a symmetric relation: if $D$ is an orthogonal dual of $V$, then $V$ is an orthogonal dual of $D$.
To introduce a concrete example, call $A \subseteq \Rspace^2$ \emph{locally finite} if every disk contains only finitely many points of $A$, and \emph{coarsely dense} if every half-plane contains infinitely many points of $A$.
For each $a \in A$, write $\domain{}{a}$ for the points $x \in \Rspace^2$ that satisfy $\Edist{x}{a} \leq \Edist{x}{b}$ for all $b \in A$, and note that $\domain{}{a}$ is a closed convex polygon.
The \emph{(order-$1$) Voronoi tessellation} of $A$---named after Georgy Voronoi \cite{Vor070809} and denoted $\Voronoi{}{A} = \Voronoi{1}{A}$---consists of the tiles $\domain{}{a}$, $a \in A$, and their edges and vertices.
The \emph{(order-$1$) Delaunay mosaic} of $A$---named after Boris Delaunay or Delone \cite{Del34} and denoted $\Delaunay{}{A} = \Delaunay{1}{A}$--- is obtained by drawing an edge connecting $a, b \in A$ whenever $\domain{}{a}$ and $\domain{}{b}$ share an edge.
These edges decompose the plane into convex polygons, which are the tiles of $\Delaunay{}{A}$.
It is well known, and also not difficult to prove that $\Delaunay{}{A}$ is an orthogonal dual of $\Voronoi{}{A}$.
If $A$ is \emph{generic}, by which we mean that no four points lie on a common circle, then every vertex of $\Voronoi{}{A}$ has degree $3$ and every tile of $\Delaunay{}{A}$ is a triangle.

\subsection{Weighted Points}
\label{sec:2.2}

We generalize the Voronoi tessellation and Delaunay mosaic to points with real weights.
For a more comprehensive treatment of this subject see \cite{Ede01}.
To begin, let $A \subseteq \Rspace^2$ be a set of unweighted points, locally finite and coarsely dense, as before.
For each $a \in A$, let $\aaaaa \colon \Rspace^2 \to \Rspace$ defined by $\aaaaa (x) = 2 \scalprod{a}{x} - \norm{a}^2$ be the corresponding \emph{affine map}, and $\aaa = (a, \norm{a}^2) \in \Rspace^2 \times \Rspace$ the corresponding \emph{lifted point}.
Let $\Polyhedron{V}{A}$ be the intersection of the half-spaces bounded from below by the graphs of the affine maps, and let $\Polyhedron{D}{A}$ be the convex hull of the lifted points in $\Rspace^3$.
Because $A$ is locally finite, both $\Polyhedron{V}{A}$ and $\Polyhedron{D}{A}$ are convex polyhedra, and because $A$ is coarsely dense, they are both unbounded, with the interior above the boundary complex, which is the graph of a piecewise linear function from $\Rspace^2$ to $\Rspace$.
The following result was at least partially known already to Voronoi \cite{Vor070809}:
\begin{proposition}[Vertical Projection]
  \label{prop:vertical_projection}
  Let $A \subseteq \Rspace^2$ be locally finite and coarsely dense.
  \medskip \begin{enumerate}[1.]
    \item $\Voronoi{}{A}$ is the vertical projection of the boundary complex of $\Polyhedron{V}{A}$ to $\Rspace^2$;
    \item $\Delaunay{}{A}$ is the vertical projection of the boundary complex of $\Polyhedron{D}{A}$ to $\Rspace^2$.
  \end{enumerate} \medskip
\end{proposition}
It is now easy to generalize the two tilings to points with weights.
Let $w \colon A \to \Rspace$ map each point to its \emph{weight},
and define $\aaaaa_w (x) = 2 \scalprod{a}{x} - \norm{a}^2 + w(a)$
and $\aaa_w = (a, \norm{a}^2 - w(a))$.
Correspondingly, $\wPolyhedron{V}{A}{w}$ is the intersection of the closed half-spaces bounded from below by the graphs of the $\aaaaa_w$, and $\wPolyhedron{D}{A}{w}$ is the convex hull of the $\aaa_w$ in $\Rspace^3$.
We call the vertical projection of the boundary complex of $\wPolyhedron{V}{A}{w}$ to $\Rspace^2$ the \emph{weighted Voronoi tessellation} of $A$ and $w$, denoted $\wVoronoi{}{A}{w}$,
and the vertical projection of the boundary complex of $\wPolyhedron{D}{A}{w}$ to $\Rspace^2$ the \emph{weighted Delaunay mosaic} of $A$ and $w$, denoted $\wDelaunay{}{A}{w}$.
An important difference to the unweighted case is that not every point in $A$ is necessarily associated with a tile in $\wVoronoi{}{A}{w}$.
Correspondingly, not every point in $A$ is also a vertex in $\wDelaunay{}{A}{w}$.

\medskip
These two tilings are known in the literature under a variety of names, including Dirichlet tessellations and power diagrams for $\wVoronoi{}{A}{w}$, and Laguerre triangulations and regular triangulations for $\wDelaunay{}{A}{w}$.
Note that $\wVoronoi{}{A}{w} = \Voronoi{}{A}$ and $\wDelaunay{}{A}{w} = \Delaunay{}{A}$ if $w(a) = 0$ for every $a \in A$.
It is not difficult to see that $\wDelaunay{}{A}{w}$ is an orthogonal dual of $\wVoronoi{}{A}{w}$ for every $w \colon A \to \Rspace$.

\subsection{Tessellations}
\label{sec:2.3}

Let $A \subseteq \Rspace^2$ be locally finite and coarsely dense.
For every finite $B \subseteq A$, we write $\domain{}{B}$ for the points $x \in \Rspace^2$ that satisfy $\Edist{x}{b} \leq \Edist{x}{a}$ for all $b \in B$ and all $a \in A \setminus B$.
Observe that $\domain{}{B}$ is a closed convex polygon, and if we drop the requirement that $A$ be coarsely dense, then the polygon may be unbounded.
Whenever $B \setminus B'$ and $B' \setminus B$ are both non-empty, $\domain{}{B}$ and $\domain{}{B'}$ have disjoint interiors but may overlap along a shared edge or vertex.

\medskip
For every $k \geq 1$, the \emph{order-$k$ Voronoi tessellation}, denoted $\Voronoi{k}{A}$, consists of all polygons $\domain{}{B}$ with cardinality $\card{B} = k$.
Two such tessellations are shown in Figure~\ref{fig:tessellations}, namely $\Voronoi{5}{\Zspace^2}$ in the upper left panel and $\Voronoi{6}{\Zspace^2}$ in the upper right panel.

For contiguous orders, $k-1$ and $k$, the two Voronoi tessellations have pairwise non-crossing edges.
Indeed, $\Voronoi{k-1}{A}$ and $\Voronoi{k}{A}$ share some of their vertices, which is where the paths of edges in the two tessellations cross.
We can therefore \emph{overlay} $\Voronoi{k-1}{A}$ and $\Voronoi{k}{A}$, by which we mean the tessellation formed by drawing all their vertices and edges.
Each tile in the overlay is the intersection of a tile, $\domain{}{B'}$ in $\Voronoi{k-1}{A}$ and another, $\domain{}{B}$ in $\Voronoi{k}{A}$, with $B' \subseteq B$.
Note that for every point in this intersection, the points in $B'$ are the $k-1$ closest, and the unique point in $B \setminus B'$ is the $k$-closest.
The collection of tiles for which a point $a \in A$ is the $k$-th closest is also known as the \emph{$k$-th Brillouin zone} of $a$.
To construct it, we draw every bisector of $a$ and another point in $A$.
This gives a line arrangement, and we select all regions that are separated from $a$ by exactly $k-1$ bisectors.
In summary, the overlay of $\Voronoi{k-1}{A}$ and $\Voronoi{k}{A}$ is the decomposition of the plane into the $k$-th Brillouin zones of the points.
This overlay was introduced in \cite{EdSe86}, where it is referred to as the \emph{degree-$k$ Voronoi tessellation}.
Because of its connection to the $k$-th Brillouin zones, we prefer to call it the \emph{order-$k$ Brillouin tessellation}, denoted $\Brillouin{k}{A}$.
As an example, consider $\Brillouin{6}{\Zspace^2}$ in the upper middle panel of Figure~\ref{fig:tessellations}.

\medskip
Three non-collinear points, $a,b,c$, define a unique circle that passes through them.
We call it the \emph{circumcircle} and denote it $\sigma = \sigma (a,b,c)$.
It \emph{encloses} the points in the open disk bounded by $\sigma$, and we write $\In{\sigma} \subseteq A$ for the points enclosed by $\sigma$, and $\On{\sigma} \subseteq A$ for the points on $\sigma$.
\begin{lemma}[Vertex Characterization]
  \label{lem:vertex_characterization}
  Let $A \subseteq \Rspace^2$ be locally finite, $\sigma$ the circumcircle of three points in $A$, and $k \geq 1$.
  \medskip \begin{enumerate}[1.]
    \item The center of $\sigma$ is a vertex of $\Voronoi{k}{A}$ for $\card{\In{\sigma}} + 1 \leq k \leq \card{\In{\sigma}} + \card{\On{\sigma}} - 1$, and the degree of this vertex in $\Voronoi{k}{A}$ is $\card{\On{\sigma}}$.
    \item The center of $\sigma$ is a vertex of $\Brillouin{k}{A}$ for $\card{\In{\sigma}} + 1 \leq k \leq \card{\In{\sigma}} + \card{\On{\sigma}}$, and the degree of this vertex in $\Brillouin{k}{A}$ is $\card{\On{\sigma}}$ whenever there is equality on the left or the right, and it is $2 \card{\On{\sigma}}$ if both inequalities are strict.
  \end{enumerate} \medskip
  The conditions exhaust the vertices of the order-$k$ Voronoi and Brillouin tessellations of $A$.
\end{lemma}
\begin{proof}
  For $x\in \Rspace^2$ and $k \geq 1$, there exists at least one set $B \subseteq A$ with $\card{B}=k$ and $x\in\domain{}{B}$. 
  Consider the smallest closed disk centered at $x$ that contains $B$ and call the boundary of the disk $\sigma$.
  The set $B$ may not be unique, but the circle $\sigma$ is.
  Consider a perturbation $y$ of $x$, such that there exists a unique $B'$ with $\card{B'}=k$ and $y\in\domain{}{B'}$.
  Assuming the perturbation is sufficiently small, $B$ consists of the points of $\In{\sigma}$ with additional $k - \card{\In{\sigma}}$ consecutive points from $\On{\sigma}$. 
  For a circle with $n = \card{\On{\sigma}}$ points, there are $n$ different sets of consecutive points of a fixed cardinality $1 \leq m \leq n-1$. 
  Thus, for $1 \leq k - \card{\In{\sigma}} \leq \card{\On{\sigma}}-1$, there are $\card{\On{\sigma}}$ different sets $B'$ and thus $\card{\On{\sigma}}$ Voronoi domains of order $k$ that meet at $x$.
  Hence, $x$ is a vertex of $\Voronoi{k}{A}$ with degree $\card{\On{\sigma}}$ for $\card{\In{\sigma}} + 1 \leq k \leq \card{\In{\sigma}} + \card{\On{\sigma}} - 1$ iff $\card{\On{\sigma}}\geq 3$.

  The second claim follows from the first because $\Brillouin{k}{A}$ is the overlay of $\Voronoi{k-1}{A}$ and $\Voronoi{k}{A}$ with pairwise non-crossing edges.
\end{proof}
For generic sets, the degrees of the vertices in the Voronoi and Brillouin tessellations are $3$ and $6$.
The integer lattice, $\Zspace^2$, is not generic, which explains why some vertices in its tessellations have degree different from $3$ and from $6$; see the top row of Figure~\ref{fig:tessellations}.
For example, $\Voronoi{5}{\Zspace^2}$ and $\Voronoi{6}{\Zspace^2}$ share vertices that have degree $8$ in both tessellations and therefore degree $16$ in $\Brillouin{6}{\Zspace^2}$, which is shown in the middle panel.

\begin{figure}[hbt]
  \centering
    \vspace{-0.0in}
    \resizebox{!}{1.6in}{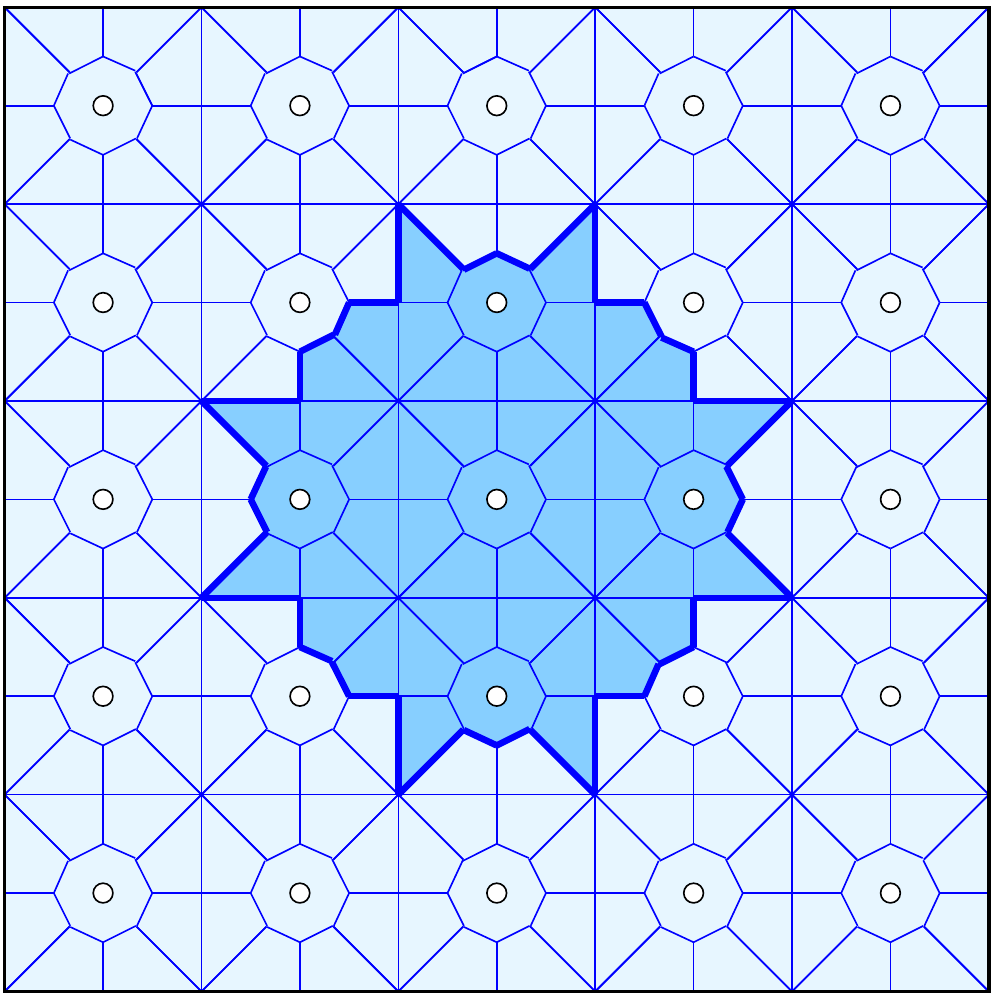}
    \hspace{0.1in}
    \resizebox{!}{1.6in}{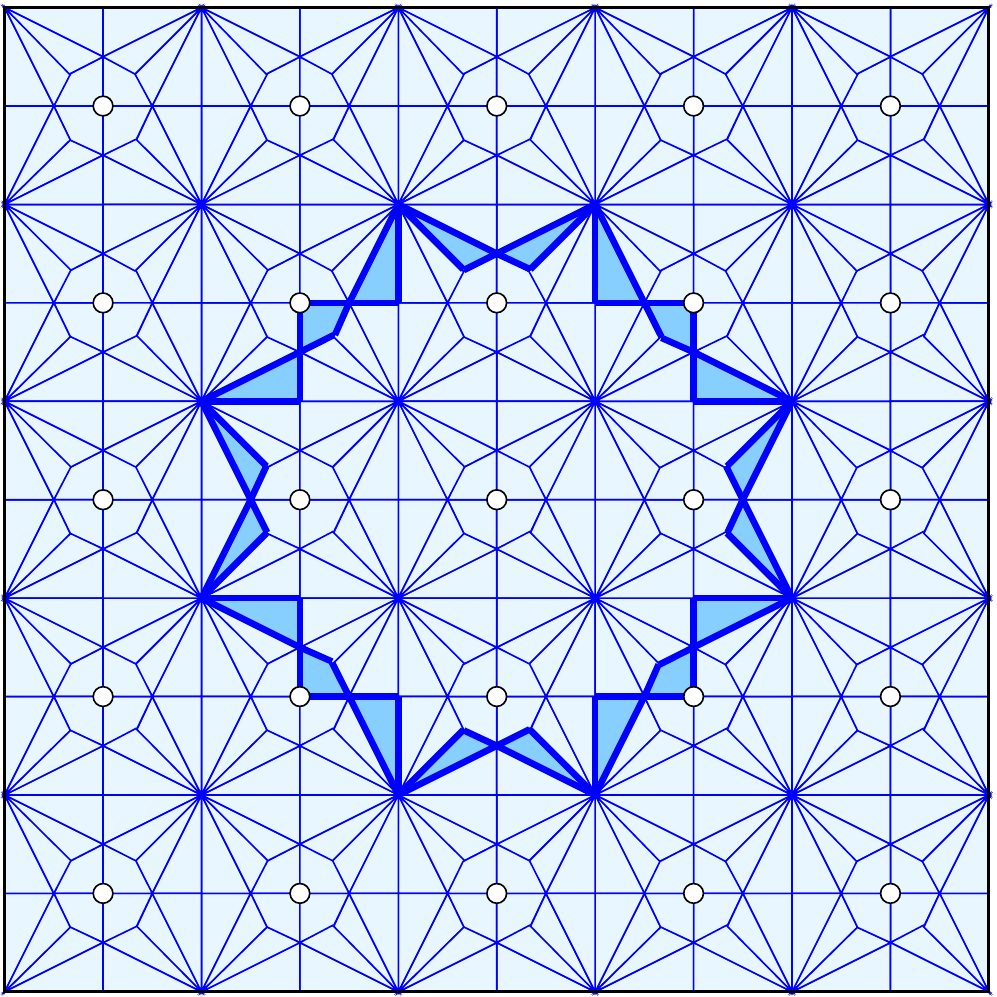}
    \hspace{0.1in}
    \resizebox{!}{1.6in}{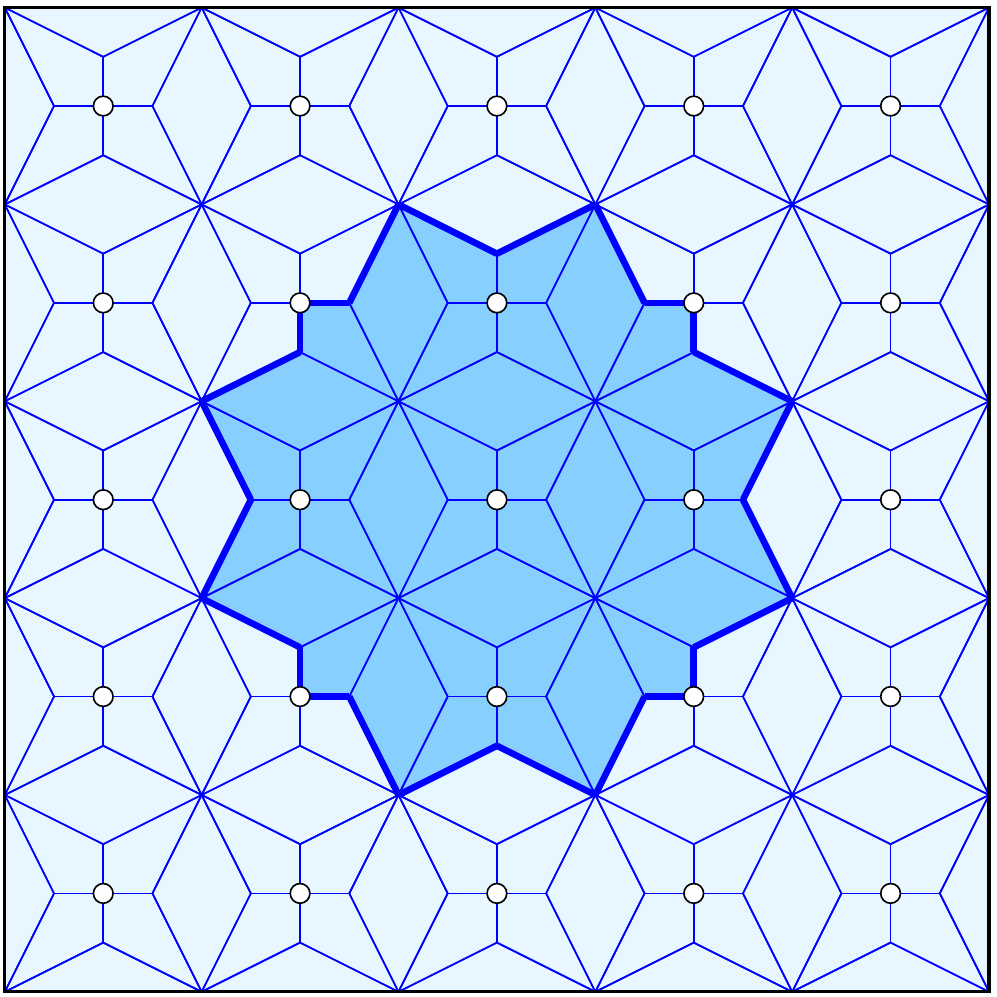} \\
    \vspace{0.15in}
    \resizebox{!}{1.6in}{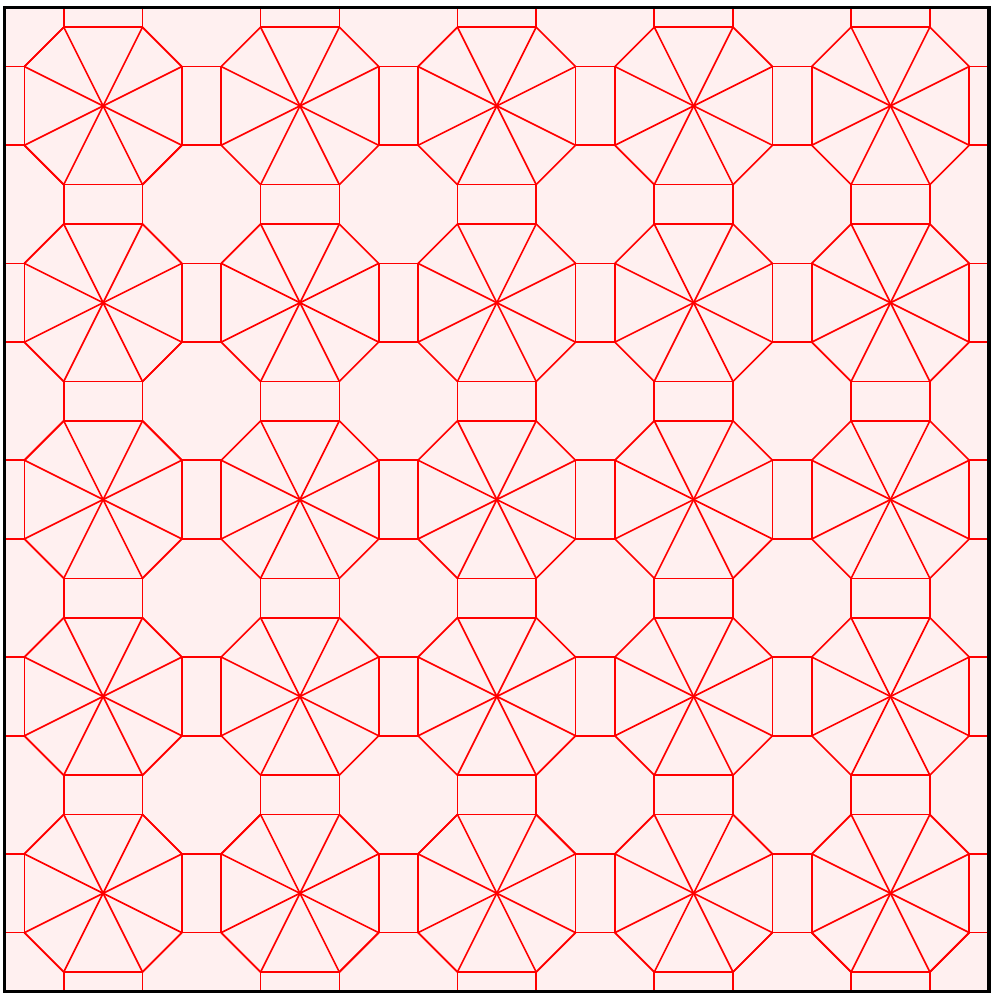}
    \hspace{0.1in}
    \resizebox{!}{1.6in}{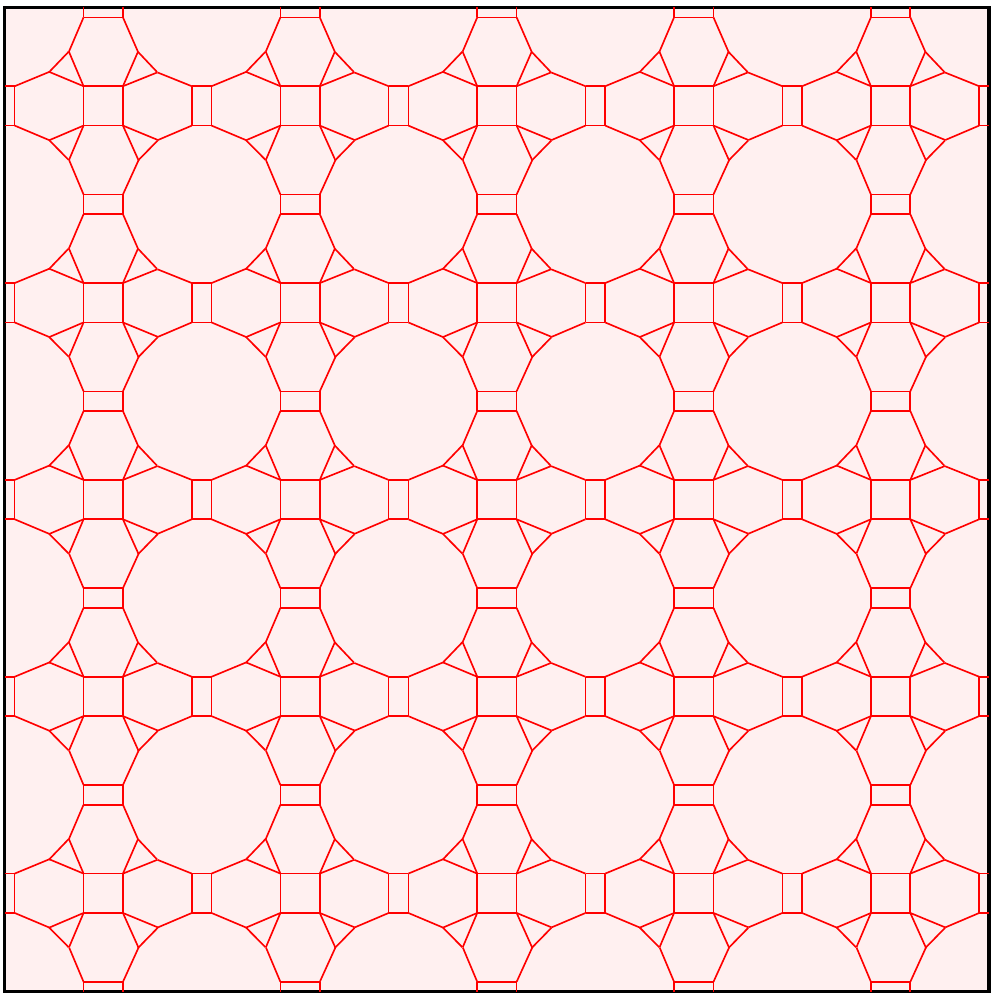}
    \hspace{0.1in}
    \resizebox{!}{1.6in}{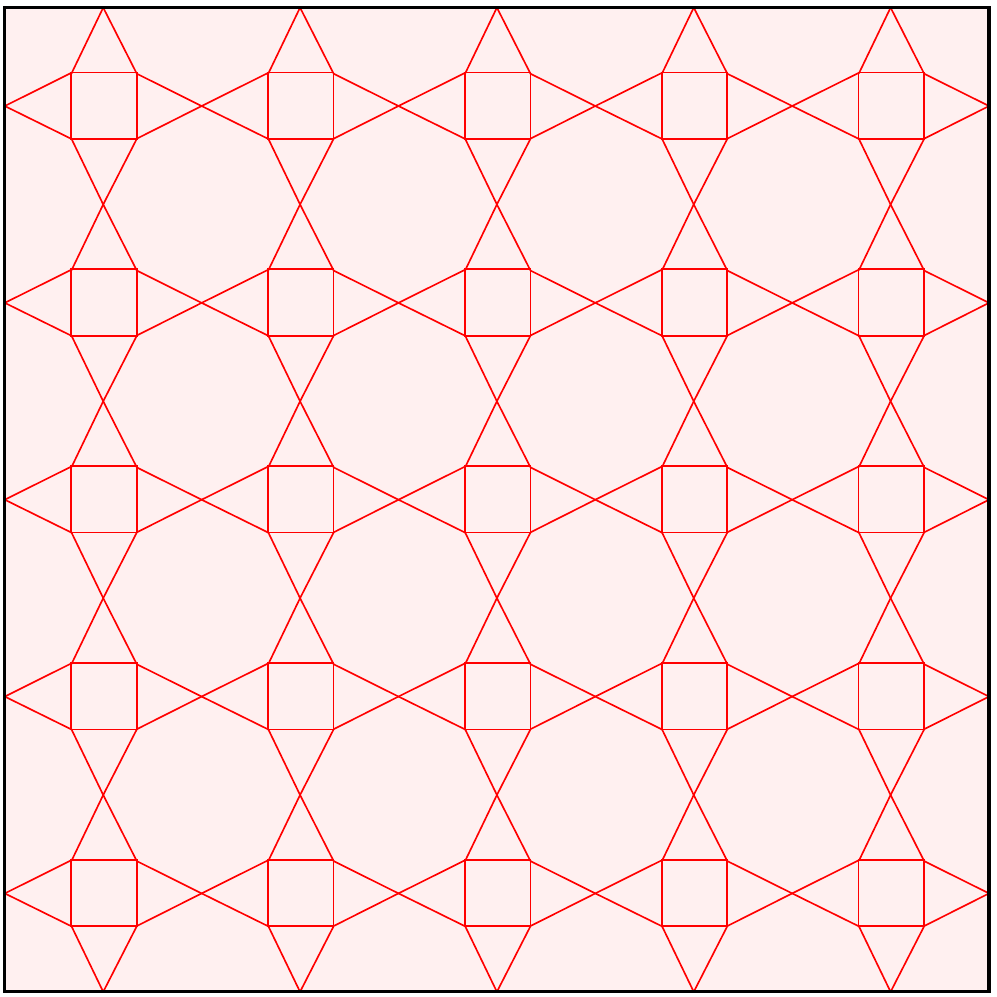} \\
    \vspace{-0.05in}
    \caption{\emph{Top row:} three tessellations of the integer lattice: $\Voronoi{5}{\Zspace^2}$ on the \emph{left}, $\Voronoi{6}{\Zspace^2}$ on the \emph{right}, and their overlay, $\Brillouin{6}{\Zspace^2}$, in the \emph{middle}.
    For every point in the \emph{dark blue} region, the point in the center is among the five closest on the \emph{left}, the sixth closest in the \emph{middle}, and among the six closest on the \emph{right}.
    \emph{Bottom row:} the corresponding mosaics: $\Delaunay{5}{\Zspace^2}$ on the \emph{left}, $\Delaunay{6}{\Zspace^2}$ on the \emph{right}, and $\Iglesias{6}{\Zspace^2}$ in the \emph{middle}.
    Observe that the mosaics are indeed orthogonal duals of the tessellations.}
  \label{fig:tessellations}
\end{figure}

\subsection{Mosaics}
\label{sec:2.4}

There is a relation between (unweighted) order-$k$ Voronoi tessellations and weighted (order-$1$) Voronoi tessellations introduced by Aurenhammer \cite{Aur90}.
Specifically, for every locally finite set $A$ and every finite $k \geq 1$, there is another locally finite set $A_k$ and function $w_k \colon A_k \to \Rspace$ such that $\Voronoi{k}{A} = \wVoronoi{}{A_k}{w_k}$.
Using this relation, we construct orthogonal duals of the order-$k$ Voronoi tessellations.
To describe the relation in detail, write
\begin{align}
  \point{B}  &=  \tfrac{1}{k} \sum\nolimits_{a \in B} a,
  \label{eqn:VD1} \\
  \height{B}  &=  \tfrac{1}{k} \sum\nolimits_{a \in B} \norm{a}^2 ,
  \label{eqn:VD2}
\end{align}
in which $k$ is the cardinality of a finite $B \subseteq A$.
Let $A_k$ be the set of points $\point{B}$, for all $B \subseteq A$ of size $k$, and let $w_k \colon A_k \to \Rspace$ be defined by $w_k (\point{B}) = \norm{\point{B}}^2 - \height{B}$.
We have $\Voronoi{k}{A} = \wVoronoi{}{A_k}{w_k}$, as proved for finite sets in \cite{Aur90}.
For locally finite and coarsely dense sets $A$, the set $A_k$ is coarsely dense but not necessarily locally finite.
It is however easy to prove that there is a locally finite subset of $A_k$ such that all points not in this subset have empty tiles in the Voronoi tessellation and are therefore irrelevant.
This suffices to show that $\wVoronoi{}{A_k}{w_k}$ is well defined.
Furthermore, $\wDelaunay{}{A_k}{w_k}$ is well defined and is an orthogonal dual of $\Voronoi{k}{A}$; see the lower left and right panels in Figure~\ref{fig:tessellations}.
A similar construction can be done for the order-$k$ Brillouin tessellation.
Write
\begin{align}
  \pointt{B}{b} &=  \tfrac{k}{2k-1} \point{B} + \tfrac{k-1}{2k-1} \point{B \setminus \{b\}}  =
  \tfrac{1}{2k-1} b + \tfrac{2}{2k-1} \sum\nolimits_{a \in B \setminus \{b\}} a ,
  \label{eqn:BI1} \\
  \heightt{B}{b} &=  \tfrac{k}{2k-1} \height{B} + \tfrac{k-1}{2k-1} \height{B \setminus \{b\}} =
  \tfrac{1}{2k-1} \norm{b}^2 + \tfrac{2}{2k-1} \sum\nolimits_{a \in B \setminus \{b\}} \norm{a}^2 ,
  \label{eqn:BI2}
\end{align}
in which $k$ is again the cardinality of $B \subseteq A$, which we assume is finite.
Observe that the coefficients are positive and add up to one.
Let $A_{k,1}$ be the set of points $\pointt{B}{b}$, for all subsets $B \subseteq A$ of size $k$ and $b \in B$, and let $w_{k,1} \colon A_{k,1} \to \Rspace$ be defined by $w_{k,1} (\pointt{B}{b}) = \norm{\pointt{B}{b}}^2 - \heightt{B}{b}$.
We have $\Brillouin{k}{A} = \wVoronoi{}{A_{k,1}}{w_{k,1}}$, as proved in \cite{EdIg16}, which implies that $\wDelaunay{}{A_{k,1}}{w_{k,1}}$ is an orthogonal dual of $\Brillouin{k}{A}$; see the lower middle panel in Figure~\ref{fig:tessellations}.
\begin{definition}[Orthogonal Dual Mosaics]
  \label{def:orthogonal_dual_mosaics}
  Noting that they are orthogonal duals of $\Voronoi{k}{A}$ and $\Brillouin{k}{A}$, we call $\Delaunay{k}{A}: = \wDelaunay{}{A_k}{w_k}$ the \emph{order-$k$ Delaunay mosaic} and $\Iglesias{k}{A} := \wDelaunay{}{A_{k,1}}{w_{k,1}}$ the \emph{order-$k$ Iglesias mosaic} of $A$; see \cite{EdIg16}.
\end{definition}
We remark that \cite{EdIg16} describes a $1$-parameter family of coefficients that generate points with real weights whose weighted order-$1$ Voronoi tessellations are the order-$k$ Brillouin tessellation of $A$.
In particular, there are two positive coefficients, $w_1 < w_0$, that satisfy $(k-1)w_0 + w_1 = 1$, and for every $B \subseteq A$ of size $k$ and $b \in B$, we use $w_1$ for $b$ and $w_0$ for every other point in $B$.
The coefficients used in \eqref{eqn:BI1} and \eqref{eqn:BI2} satisfy these conditions, and they are special as they imply centrally symmetric hexagons in the order-$k$ Iglesias mosaic for generic $A$ (see Section~\ref{sec:3.1}), which the other choices do not.
In other words, up to scaling and translation, $\Iglesias{k}{A}$ is the unique member in a $1$-parameter family of orthogonal duals of $\Brillouin{k}{A}$ that guarantees centrally symmetric hexagons in the generic case.

\section{Monotonicity of Angles}
\label{sec:3}

In preparation of the main theorem, we take a detailed look at the angles we find in the mosaics and tessellations.

\subsection{Angle Types}
\label{sec:3.1}

Assuming $A$ is generic, all vertices in $\Voronoi{k}{A}$ have degree $3$, and we distinguish between \emph{old vertices}, which it shares with $\Voronoi{k-1}{A}$, and \emph{new vertices}, which it shares with $\Voronoi{k+1}{A}$.
Similarly, $\Brillouin{k}{A}$ has three age-groups of vertices depending on the Voronoi tessellations it shares the vertex with: \emph{old} for orders $k-2, k-1$, \emph{mid} for orders $k-1,k$, and \emph{new} for orders $k, k+1$.
Since $\Brillouin{k}{A}$ is the overlay of $\Voronoi{k-1}{A}$ and $\Voronoi{k}{A}$, its old and new vertices have degree $3$ and its mid vertices have degree $6$.
Correspondingly, we have \emph{old} and \emph{new tiles} in $\Delaunay{k}{A}$ and \emph{old}, \emph{mid}, \emph{new tiles} in $\Iglesias{k}{A}$.
In the generic case, the old and new tiles are triangles, and the mid tiles are hexagons.

To be specific about the tiles, let $a, b, c \in A$ with $\card{\In{\sigma}} = \ell$, in which $\sigma$ is the circumcircle of the three points.
According to Lemma~\ref{lem:vertex_characterization}, the center of $\sigma$ is new, old in $\Voronoi{\ell+1}{A}, \Voronoi{\ell+2}{A}$, and new, mid, old in $\Brillouin{\ell+1}{A}, \Brillouin{\ell+2}{A}, \Brillouin{\ell+3}{A}$, respectively.
In the dual Delaunay and Iglesias mosaics, $a,b,c$ define triangles and hexagons whose vertices are specified in \eqref{eqn:VD1} and \eqref{eqn:BI1}.
Write $u$ for the sum of points in $\In{\sigma}$, let $x \in \{a,b,c\}$, let $y \in \{a,b,c\} \setminus \{x\}$, and let $z \in \{a,b,c\} \setminus \{x,y\}$.
Then the tiles defined by $a,b,c$ are:
\medskip \begin{itemize}
  \item in $\Delaunay{\ell+1}{A}$:  the new triangle with vertices $\frac{1}{\ell+1} (u+x)$;
  \item in $\Delaunay{\ell+2}{A}$:  the old triangle with vertices
  $\frac{1}{\ell+2} (u+x+y)$;
  \item in $\Iglesias{\ell+1}{A}$:  the new triangle with vertices $\frac{1}{2\ell+1} (2u+x)$;
  \item in $\Iglesias{\ell+2}{A}$:  the mid hexagon with vertices $\frac{1}{2\ell+3} (2u+2x+y)$;
  \item in $\Iglesias{\ell+3}{A}$:  the old triangle with vertices $\frac{1}{2\ell+5} (2u+2x+2y+z)$;
\end{itemize} \medskip
see Figure~\ref{fig:polygons} for the case in which $a,b,c$ are the vertices of an equilateral triangle.
\begin{figure}[hbt]
  \centering
    \vspace{-0.0in}
    \resizebox{!}{1.6in}{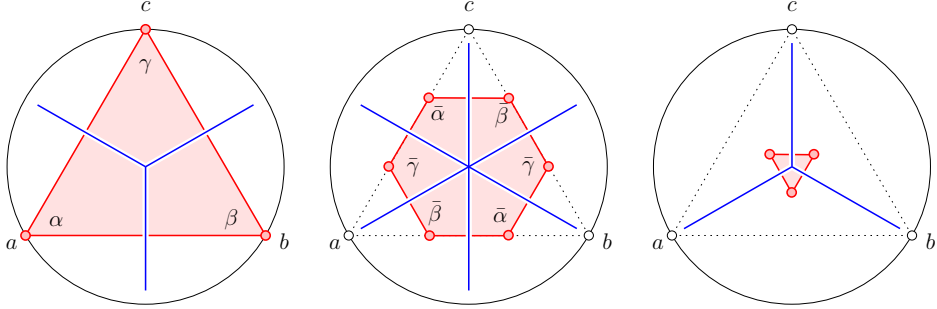} \\
    \vspace{-0.05in}
    \caption{For zero enclosed points, we get the triangle $abc$ in $\Iglesias{1}{A}$, the hexagon constructed by equal tri-sections of the edges of $abc$ in $\Iglesias{2}{A}$, and a translate of $abc$ scaled by $- \frac{1}{5}$
    in $\Iglesias{3}{A}$.}
  \label{fig:polygons}
\end{figure}
Observe that $\frac{1}{2\ell+3} (2u + a+b+c)$ is the center with respect to which the hexagon is centrally symmetric.
Indeed, we can pair up the six vertices so that the average of each pair is the center.
To illustrate the constructions, assume first that $\In{\sigma} = \emptyset$, which implies $u = 0$.
The corresponding triangle in $\Delaunay{1}{A}$ has vertices $a,b,c$, the triangle in $\Delaunay{2}{A}$ has vertices $\frac{1}{2}(a+b)$, $\frac{1}{2}(a+c)$, $\frac{1}{2}(b+c)$,
the triangle in $\Iglesias{1}{A}$ has vertices $a,b,c$, the hexagon in $\Iglesias{2}{A}$ has vertices $\frac{2}{3} a + \frac{1}{3} b$, $\frac{2}{3} a + \frac{1}{3} c$, $\frac{2}{3} b + \frac{1}{3} a$, $\frac{2}{3} b + \frac{1}{3} c$, $\frac{2}{3} c + \frac{1}{3} a$, $\frac{2}{3} c + \frac{1}{3} b$, and the triangle in $\Iglesias{3}{A}$ has vertices $\frac{2}{5}(a+b) + \frac{1}{5}c$, $\frac{2}{5}(a+c) + \frac{1}{5}b$, $\frac{2}{5}(b+c) + \frac{1}{5}a$;
see Figure~\ref{fig:polygons}.
Importantly, the two triangles are similar and thus have the same three angles, and the six angles of the hexagon are supplementary to the angles of the two triangles.
In the more general case, when $\In{\sigma}$ is not necessarily empty, the triangles and hexagons are scaled and translated copies of the shapes we see for $\In{\sigma} = \emptyset$.
Everything we said about angles thus still applies.

\medskip
An \emph{angle} at a vertex inside a convex polygon is a real number between $0$ and $\pi$.
The \emph{supplementary angle} is $\bar{\varphi} = \pi - \varphi$.
We use Lemma~\ref{lem:vertex_characterization} to decide in which tessellations and mosaics an angle or its supplement appear.
\begin{lemma}[Angles and Supplementary Angles]
  \label{lem:angles_and_supplementary_angles}
  Let $A \subseteq \Rspace^2$ be locally finite and generic, and let $\sigma$ be the circumcircle of $a,b,c \in A$ with $\card{\In{\sigma}} = \ell$ and $\varphi = \angle acb$.
  Then
   \medskip \begin{itemize}
    \item $\varphi$ is an angle in $\Delaunay{\ell+1}{A}$, $\Delaunay{\ell+2}{A}$, $\Iglesias{\ell+1}{A}$, $\Iglesias{\ell+3}{A}$, and  $\Brillouin{\ell+2}{A}$,
    \item $\bar{\varphi}$ is an angle in $\Voronoi{\ell+1}{A}$, $\Voronoi{\ell+2}{A}$, $\Brillouin{\ell+1}{A}$, $\Brillouin{\ell+3}{A}$, and $\Iglesias{\ell+2}{A}$.
  \end{itemize} \medskip
  The conditions exhaust the angles appearing in the mosaics and tessellations of $A$.
\end{lemma}
\begin{proof}
  The above considerations 
  show that translated and scaled copies of the triangle $abc$ belong to $\Delaunay{\ell+1}{A}$ and $\Iglesias{\ell+1}{A}$, centrally reflected and scaled copies belong to $\Delaunay{\ell+2}{A}$ and $\Iglesias{\ell+3}{A}$, and a hexagon containing two copies of each of the triangle's supplementary angles belongs to $\Iglesias{\ell+2}{A}$. This implies the claims about Delaunay and Iglesias mosaics. Orthogonal duality yields supplementary angles for the Voronoi and Brillouin tessellations and thus the remaining claims.
  
  Due to Lemma~\ref{lem:vertex_characterization}, the above conditions exhaust the angles appearing in the mosaics and tessellations of $A$.
\end{proof}

\subsection{Monotonicity Theorem}
\label{sec:3.2}

We prepare the proof of the main theorem with a technical lemma.
For any three non-collinear points, $a,b,c \in A$, write $\sigma = \sigma (a,b,c)$ for the unique circle that passes through the points, and $\Angle{acb}$ for the angle at $c$ inside the triangle with vertices $a,b,c$.
Recall that ${\In{\sigma}}$ are the points of $A$ that lie in the open disk bounded by $\sigma$.
Assuming a generic set, we define
\begin{align}
  \alpha_\ell(A)  &=  \inf \{ \Angle{acb} \mid a,b,c \in A, \card{\In{\sigma(a,b,c)}} = \ell \};
    \label{eqn:infangle} \\
  \beta_\ell(A)  &=  \inf \{ \pi-\Angle{acb} \mid a,b,c \in A, \card{\In{\sigma(a,b,c)}} = \ell \}.
    \label{eqn:infsuppangle}
\end{align}
We define $\alpha_\ell (A)$ and $\beta_\ell (A)$ also for possibly non-generic sets, but here we count an angle whenever there is an arbitrarily small perturbation such that the angle is counted for the now generic set.
For example, if $\sigma = \sigma(a,b,c)$ passes through $n+1$ points and encloses $p$ points, then $\Angle{abc}$ is counted for $p \leq \ell \leq p+(n-2)$.
Indeed, after fixing the three points that define the angle, $\ell - p$ of the $n-2$ remaining points on the circle may join the $p$ points inside the circle.
\begin{lemma}[Angle Inequalities]
  \label{lem:angle_inequalities}
  Let $A \subseteq \Rspace^2$ be locally finite and coarsely dense.
  Then $\beta_\ell (A) \geq \alpha_\ell (A) \geq \alpha_{\ell+1} (A)$, for all $\ell \geq 0$.
\end{lemma}
\begin{proof}
  The supplement of the largest angle in a triangle is the sum of the other two angles and therefore necessarily larger than the smallest angle.
  We thus get $\beta_\ell (A) \geq \alpha_\ell (A)$.
  
  \medskip
  We prove the second inequality first in the generic case.
  Let $\sigma_0$ be the circle passing through $a,b,c \in A$, with $\card{\In{\sigma_0}} = \ell$ and $\Angle{acb} \leq \pi /3$.
  Consider the pencil of circles that pass through $a$ and $b$, of which $\sigma_0$ is one member.
  The circles in this pencil that enclose $c$ are necessarily larger than $\sigma_0$.
  Among these, let $\sigma_1, \sigma_2, \ldots$ be the circles that pass through a third point, which we list in the order of increasing radius.
  The points $a, b$ decompose each $\sigma_i$ into two arcs, of which the \emph{short arc} lies inside $\sigma_0$ and the \emph{long arc} lies outside $\sigma_0$.
  Let $c_i \in A$ be the third point on $\sigma_i$, after $a$ and $b$.
  We say $\sigma_i$ \emph{decrements} if $c_i$ lies on the short arc: $\card{\In{\sigma_i}}$ is one less than the count for the circles between $\sigma_{i-1}$ and $\sigma_i$.
  Symmetrically, we say $\sigma_i$ \emph{increments} if $c_i$ lies on the long arc: the count for the circles between $\sigma_{i}$ and $\sigma_{i+1}$ is one greater than $\card{\In{\sigma_i}}$.
  Since $A$ is locally finite, there are only finitely many decrementing circles, and since $A$ is coarsely dense, there are infinitely many incrementing circles.
  This implies that there exists an index $j \geq 1$ such that $\sigma_j$ increments and $\card{\In{\sigma_j}} = \ell+1$.
  By the Inscribed Angle Theorem, $\angle{a c_j b} < \angle{acb}$.
  Since we find a smaller angle for every triplet $a,b,c \in A$ with $\card{\In{\sigma_0}} = \ell$ and $\angle{acb} \leq \pi/3$, we conclude that $\alpha_{\ell+1} \leq \alpha_\ell$.
  
  \medskip
  We prove the remaining non-generic case by contradiction.
  Suppose $\alpha_\ell (A) < \alpha_{\ell+1} (A)$ for some $\ell \geq 0$.
  Then there exist points $a,b,c \in A$ with $\Angle{acb} < \alpha_{\ell+1} (A)$ such that $p+(n-2) = \ell$, in which $p = \card{\In{\sigma}}$, $n+1 = \card{\On{\sigma}}$, and $\sigma = \sigma(a,b,c)$.
  Note that $p + (n-2) > \ell$ is not possible, else a perturbation could have $\ell+1$ points inside the circle, which contradicts that $\angle{acb}$ is strictly smaller than $\alpha_{\ell+1}(A)$.
  There exists an arbitrarily small perturbation $A'$ of $A$ such that $a,b,c \in A'$ and $\sigma$ encloses $\ell$ points, namely the perturbed images of the $p$ points in $\In{\sigma}$ and $\ell-p$ of the $n-2$ points in $\On{\sigma} \setminus \{a,b,c\}$.
  The perturbation can be chosen arbitrarily small so that $A'$ is generic and $\alpha_{\ell+1} (A')$ is arbitrarily close to $\alpha_{\ell+1} (A)$ and therefore strictly larger than $\angle{acb}$.
  To see the latter property, consider the straight-line homotopy from $A$ to $A'$ and observe that $\alpha_{\ell+1}$ depends continuously on the parameter that controls the homotopy.
  Indeed, the Inscribed Angle Theorem guarantees that every angle, $\angle{xzy}$, that contributes to $\alpha_{\ell+1}$ appears at least $\card{\On{\sigma(x,y,z)}}-2$ times.
  After a small perturbation, at least one of these angles still contributes to $\alpha_{\ell+1}$.
\end{proof}

We note that Lemma~\ref{lem:angle_inequalities} does not generalize to finite sets.
To see this, let $A$ be the vertices of an equilateral triangle, $a, b, c$, together with the barycenter, $d = \tfrac{1}{3}(a+b+c)$.
The four points define four circles of which three enclose no point and the circle that passes through $a,b,c$ encloses one point.
Hence, $\pi/6 = \alpha_0 < \alpha_1 = \pi/3$, which contradicts the second inequality in Lemma~\ref{lem:angle_inequalities} for finite sets.
To generalize, we place an additional $k$ points near $d$, which gives $\alpha_{k+1} = \pi/3$ and $\alpha_k$ as close to $\pi/6$ as we like.

\medskip
Write $\alpha (\Delaunay{k}{A})$ and $\omega (\Delaunay{k}{A})$ for the infimum and supremum angles in the order-$k$ Delaunay mosaic, and similarly for the other mosaics and the tessellations in this paper.
We prove that $\alpha$ and $\omega$ behave mostly monotonically, but of course not for finite sets for which even Lemma~\ref{lem:angle_inequalities} does not hold.
\begin{theorem}[Monotonicity of Extreme Angles]
  \label{thm:monotonicity_of_extreme_angles}
  Let $A \subseteq \Rspace^2$ be locally finite, coarsely dense, and generic, and let $k \geq 1$.
  Then
  \medskip \begin{enumerate}[1.]
    \item $\alpha (M_k(A)) \geq \alpha (M_{k+1}(A))$, with $M \in \{\mbox{\rm Del, Igl, Bri}\}$,
    \item $\omega (M_k(A)) \leq \omega (M_{k+1}(A))$, with $M \in \{\mbox{\rm Vor, Bri, Igl}\}$.
  \end{enumerate} \medskip
  Furthermore, $\alpha (\Brillouin{k}{A}) \geq \alpha (\Brillouin{k+1}{A})$ and $\omega (\Iglesias{k}{A}) \leq \omega (\Iglesias{k+1}{A})$ even if we drop the requirement that $A$ be generic.
\end{theorem}
\begin{proof}
  We have $\pi - \omega (\Voronoi{k}{A}) = \alpha (\Delaunay{k}{A} )$ since the Voronoi tessellation is an orthogonal dual of the Delaunay mosaic.
  Similarly, $\pi - \omega (\Brillouin{k}{A}) = \alpha (\Iglesias{k}{A} )$ and $\pi - \omega (\Iglesias{k}{A}) = \alpha (\Brillouin{k}{A} )$.
  It thus suffices to prove the three inequalities for the infimum angles.
  In the generic case, we have
  \begin{align}
    \alpha (\Delaunay{k}{A}) = \min \{\alpha_{k-2}, \alpha_{k-1}\}  &\geq  \min \{\alpha_{k-1}, \alpha_k\} = \alpha (\Delaunay{k+1}{A}) , 
      \label{eqn:first} \\
    \alpha (\Iglesias{k}{A}) = \min \{\alpha_{k-3}, \beta_{k-2}, \alpha_{k-1}\}  
    &\geq 
    \min \{\alpha_{k-2}, \beta_{k-1}, \alpha_k\} = \alpha (\Iglesias{k+1}{A}) , 
      \label{eqn:second} \\
    \alpha (\Brillouin{k}{A}) = \min \{\beta_{k-3}, \alpha_{k-2}, \beta_{k-1}\}  &\geq  \min \{\beta_{k-2}, \alpha_{k-1}, \beta_k\} = \alpha (\Brillouin{k+1}{A}),
      \label{eqn:third}
  \end{align}
  in which we use Lemma~\ref{lem:angle_inequalities} to get the inequality in each of the three cases.
  For example, we use $\alpha_{k-3} \geq \alpha_{k-2}$, $\beta_{k-2} \geq \alpha_{k-2}$, and $\alpha_{k-1} \geq \alpha_k$ to establish \eqref{eqn:second}.
  No angle is compared with $\beta_{k-1}$, but this is not necessary.
  
  \medskip
  In the non-generic case, \eqref{eqn:first} and \eqref{eqn:second} fail fatally, while \eqref{eqn:third} can be rescued in a weaker form that is still sufficient to prove the claimed inequalities.
  Specifically,
  \begin{align}
    \alpha (\Brillouin{k}{A}) \geq \min \{\beta_{k-3}, \alpha_{k-2}, \beta_{k-1}\}  &\geq  \alpha_{k-1} \geq \alpha (\Brillouin{k+1}{A}),
      \label{eqn:fourth}
  \end{align}
  as we are about to prove.
  Recall that $\Brillouin{k}{A}$ consists of the $k$-th Brillouin zones of the points in $A$.
  To construct the $k$-th Brillouin zone of $a \in A$, we draw all bisectors defined by $a, b$, with $b \in A \setminus \{a\}$, and we collect all chambers in the resulting line arrangement that are separated from $a$ by exactly $k-1$ bisectors.
  For a vertex $u$ of this zone, write $\sigma$ for the circle with center $u$ and radius $\Edist{u}{a}$.
  For each $b \in \In{\sigma}$, the bisector of $a, b$ separates $a$ from $u$, and for each $b \in \On{\sigma} \setminus \{a\}$, the bisector of $a, b$ passes through $u$.
  Let $p = \card{\In{\sigma}}$ and $n+1 = \card{\On{\sigma}}$, and write $a=a_0, a_1, \ldots, a_n$ for the points in $\On{\sigma}$, listed in a counterclockwise order around $\sigma$.
  The $n$ bisectors that pass through $u$ form $2n$ angles, which we enumerate in a clockwise order as
  \begin{align}
      \pi - \angle{a_naa_1}, \angle{a_1aa_2}, \ldots, \angle{a_{n-1}aa_n},
      \pi - \angle{a_naa_1}, \angle{a_1aa_2}, \ldots, \angle{a_{n-1}aa_n}.
  \end{align}
  These are angles inside the $k$-th Brillouin zone of $a$, for $k = p+1, p+2, \ldots, p+n, p+n+1, p+n, \ldots, p+2$, in this order.
  Consider the $n+1$ cyclic rotations of the ordered list of points in $\On{\sigma}$.
  By the Inscribed Angle Theorem, any two contiguous points form the same angle at every third point.
  We have $n+1$ such angles, and each appears in $\Brillouin{k}{A}$ for $p+2 \leq k \leq p+n$.
  Furthermore, each of these angles contributes to the definition of $\alpha_\ell$ for $p \leq \ell \leq p+n-2$.
  This implies $\alpha_{k-2} \geq \alpha (\Brillouin{k}{A})$ for all $k \geq 2$, which is the third inequality in \eqref{eqn:fourth}.
  In fact, we have equality, unless the supplementary angles defined by three consecutive points along the circle are smaller.
  Here we observe a different pattern:  $\pi - \angle{a_i a_{i+1} a_{i+2}}$ belongs to $\Brillouin{k}{A}$ only for $k= p+1$ and for $k = p+n+1$ (writing indices modulo $n+1$).
  However, it contributes to the definition of $\beta_{\ell}$ for $p \leq \ell \leq p+n-2$.
  
  \medskip
  In summary, every angle that appears in $\Brillouin{k}{A}$ belongs to one of three cases.
  Writing $\sigma$ for the circle that passes through the three points that define the angle, $a_0, a_1, \ldots, a_n$ for the ordered list of points on $\sigma$, and $p = \card{\In{\sigma}}$, as before, the cases are:
  \medskip \begin{itemize}
      \item an angle of the form $\angle{a_i a_{i+1} a_j}$, with $p+2 \leq k \leq p+n$, which contributes to $\alpha_{k-2}$;
      \item an angle of the form $\pi - \angle{a_i a_{i+1} a_{i+2}}$, with $k=p+1$, which contributes to $\beta_{k-1}$;
      \item an angle of the form $\pi - \angle{a_i a_{i+1} a_{i+2}}$, with $k=p+n+1$, which contributes to $\beta_{k-3}$.
  \end{itemize} \medskip
  Therefore $\alpha (\Brillouin{k}{A}) \geq \min \{ \beta_{k-3}, \alpha_{k-2}, \beta_{k-1} \}$, which is the first inequality in \eqref{eqn:fourth}.
  We get $\min \{ \beta_{k-3}, \alpha_{k-2}, \beta_{k-1} \} \geq \alpha_{k-1}$ from Lemma~\ref{lem:angle_inequalities}, which is the middle inequality in \eqref{eqn:fourth}.
  This implies the claimed inequalities in the non-generic case.
\end{proof}

If we drop the genericity requirement for $A$, then the first inequality in Theorem~\ref{thm:monotonicity_of_extreme_angles} fails for the Delaunay and Iglesias mosaics.
See for example the Delaunay mosaics in Figure~\ref{fig:tessellations}, where the minimum angle in $\Delaunay{5}{\Zspace^2}$ is less than in $\Delaunay{6}{\Zspace^2}$.
Equivalently, the second inequality fails for the Voronoi and the Brillouin tessellations; see again Figure~\ref{fig:tessellations}.
Even with genericity assumption, the first inequality does not hold for the Voronoi tessellations and, equivalently, the second inequality does not hold for the Delaunay mosaics; see Appendix~\ref{app:A}.

\subsection{Angle Monotonicity Experimentally}
\label{sec:3.3}

Figure~\ref{fig:lattice} shows the sequences of min and max angles in the order-$k$ Delaunay mosaics and Brillouin tessellations of a lattice and a random periodic set.
The lattice is designed so that any circle that passes through more than three points encloses more points than the values of $k$ considered in our experiment.
The periodic set is constructed as $A = A_0 + \Zspace^2$, in which $A_0$ is a finite set of points chosen uniformly at random in the unit square.
For the purpose of our computational experiments, it suffices to copy $A_0$ into the eight squares surrounding the unit square.
We then collect angles spanned by three points whose circumcenter lies in the middle square, making sure that $k$ is small enough so that the circumcircles are contained in the union of the nine squares.
\begin{figure}[hbt]
  \centering
    \vspace{0.1in}
    \includegraphics[width=0.46\textwidth]{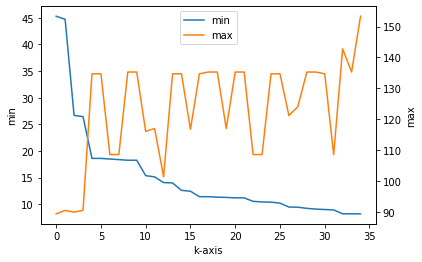}
    \includegraphics[width=0.46\textwidth]{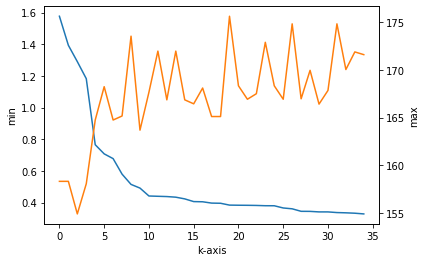}
    \includegraphics[width=0.458\textwidth]{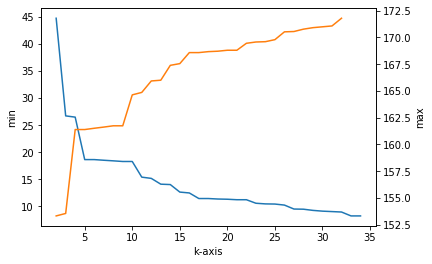}
    \includegraphics[width=0.46\textwidth]{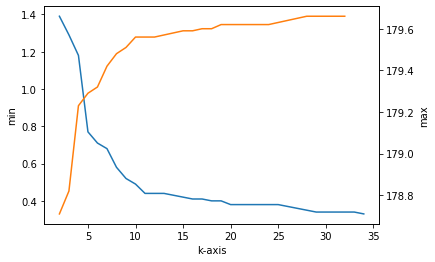}
    \vspace{-0.1in}
    \caption{\emph{Top:} the sequences of min angles (in \emph{blue}) and of max angles (in \emph{orange}) in the order-$k$ Delaunay mosaics of a lattice without four cocircular points on the \emph{left} and of a random periodic set on the \emph{right}.
    We draw the vertical axes for the minima and maxima on opposite sides of each panel.
    Note that the min angles decrease monotonically with increasing $k$, but the max angles are not monotonic.
    \emph{Bottom:} the sequences of min and max angles in the order-$k$ Brillouin tessellations of the lattice on the \emph{left} and the random periodic set on the \emph{right}.
    Both the min and the max angles are monotonic in $k$.}
  \label{fig:lattice}
\end{figure}

We note that a random periodic set is a grossly deficient approximation of a Poisson point process in $\Rspace^2$ if we are interested in infimum and supremum angles.
Indeed, for every $k \geq 1$, the expected infimum angle of a Poisson point process is $0$ just because the probability that three points form an angle smaller than $\ee$ is non-zero for every $\ee > 0$; see also Section~\ref{sec:4}.
Symmetrically, the expected supremum angle is $\pi$.
The sequences of infimum and supremum angles for a Poisson point process are thus very non-interesting.
They are however interesting for \emph{Delone sets},
which are defined by having positive packing radius and bounded covering radius.
These two radii guarantee that the angles are bounded away from $0$ and from $\pi$.

\section{Distribution of Angles}
\label{sec:4}

A stationary Poisson point process in $\Rspace^2$ satisfies the requirements of Theorem~\ref{thm:monotonicity_of_extreme_angles} with probability $1$.
Among other things, this implies that with probability $1$ the infimum angle in the order-$k$ Delaunay mosaic is non-increasing for increasing $k$. 
On the other hand, we will see that the angle distributions for different values of $k$ are the same.
This distribution is positive over the entire open interval of angles, which suggests that the infimum angle vanishes with probability $1$ for all finite orders.

\subsection{Angles in Poisson--Delaunay Mosaics}
\label{sec:4.1}

We recall that a \emph{stationary} \emph{Poisson point process} with \emph{intensity} $\intensity > 0$ is characterized by the expected number of sampled points in a Borel set of given measure, and the independence of these numbers for disjoint Borel sets.
Let $A \subseteq \Rspace^2$ be such a process.
With probability $1$, $A$ is locally finite, coarsely dense, and generic.
Because $A$ is locally finite, its order-$k$ Delaunay mosaics are defined, because $A$ is coarsely dense, they cover the entire $\Rspace^2$, and because $A$ is generic, they are simplicial.
Assuming $A$ is a Poisson point process, we call $\Delaunay{k}{A}$ an \emph{order-$k$ Poisson--Delaunay mosaic}, and similarly for $\Voronoi{k}{A}$, $\Iglesias{k}{A}$, and $\Brillouin{k}{A}$.

\medskip
An old result by Roger Miles \cite{Mil70} asserts that the distribution of the angles in the (order-$1$) Delaunay mosaic of a stationary Poisson point process in $\Rspace^2$ is
\begin{align}
  f(t)  &=  \tfrac{4}{3} [ (\pi-t) \cos t + \sin t ] \sin t .
\end{align}
Accordingly, the distribution of the supplementary angles is $g(t) = f(\pi-t)$, and the distribution of the angles together with their supplementary angles is
\begin{align}
  h(t)  &=  \tfrac{1}{2} [f(t) + g(t)]
         =  \tfrac{2}{3} [(\pi-2t) \cos t + 2 \sin t] \sin t ;
\end{align}
see Figure~\ref{fig:PPP-angles}.
The second derivative of the latter is $h''(t)  =  -\tfrac{8}{3}(\pi-2t)\sin t \cos t$, which is zero at $t = 0, \sfrac{\pi}{2}, \pi$ and negative everywhere else in $[0, \pi]$.
It follows that the distribution of angles together with their supplementary angles is concave.
\begin{figure}[hbt]
  \centering
    \vspace{0.1in}
    \includegraphics[width=0.60\textwidth]{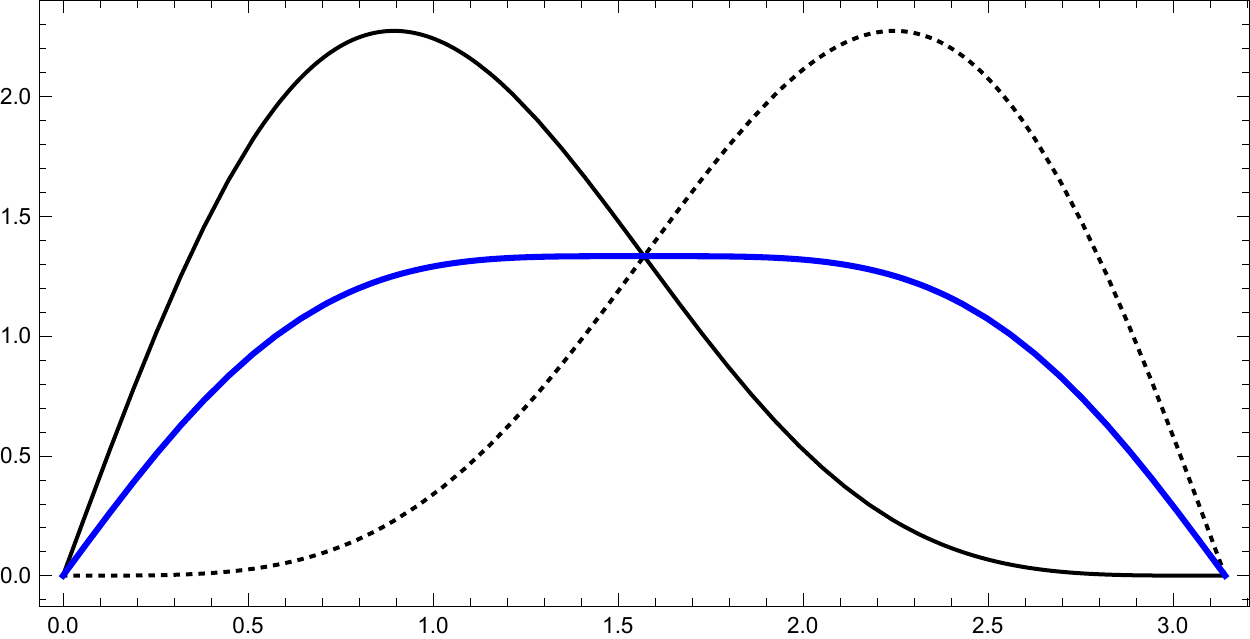} \\
    \vspace{-0.05in}
    \caption{The graphs of $f$, $g$, and $h$ in \emph{solid black}, \emph{dotted black}, and \emph{blue}.}
  \label{fig:PPP-angles}
\end{figure}

\subsection{Angles and Supplements}
\label{sec:4.2}

The main result in this section is that the three distributions displayed in Figure~\ref{fig:PPP-angles} cover all mosaics and tessellations of a stationary Poisson point process considered in this paper.
\begin{theorem}[Angle Distributions]
  \label{thm:angle_distributions}
  Let $A \subseteq \Rspace^2$ be a stationary Poisson point process and $k \geq 1$.
  Then
  \medskip \begin{enumerate}[1.]
    \item $f$ is the distribution of angles in $\Delaunay{k}{A}$,
    \item $g$ is the distribution of angles in $\Voronoi{k}{A}$,
    \item $h$ is the distribution of angles in $\Iglesias{k}{A}$ as well as of the angles in $\Brillouin{k}{A}$.
  \end{enumerate} \medskip
\end{theorem}
\begin{proof}
  By Miles' result \cite{Mil70}, $f$ is the distribution of angles in $\Delaunay{1}{A}$, and because the order-$1$ Voronoi tessellation is an orthogonal dual of the order-$1$ Delaunay mosaic, $g$ is the distribution of angles in $\Voronoi{1}{A}$.
  As proved in \cite{ENR17}, the shape of a triangle spanned by three points in a stationary Poisson point process is independent of the number of points its circumcircle encloses.
  This implies that $f$ and $g$ are the distributions of the angles in $\Delaunay{k}{A}$ and $\Voronoi{k}{A}$ for all positive $k$.
  
  \medskip
  Assuming the intensity of the Poisson point process is $\rho > 0$, the densities of old and new vertices in $\Voronoi{k}{A}$ are $(2k-1) \rho$ and $2k \rho$, respectively (see the inductive counting argument for finite sets in \cite{Lee82}, which can be adapted to the Poisson point process case).
  The degree-$3$ vertices in $\Brillouin{k}{A}$ are the old vertices in $\Voronoi{k-1}{A}$ plus the new vertices in $\Voronoi{k}{A}$, so their density is $(4k-2) \rho$.
  The degree-$6$ vertices in $\Brillouin{k}{A}$ are the new vertices in $\Voronoi{k-1}{A}$, which are also the old vertices in $\Voronoi{k}{A}$, so their density is $(2k-1) \rho$.
  It follows that the density of angles around degree-$3$ vertices is the same as around degree-$6$ vertices, namely $(12k-6) \rho$.
  The distribution of these angles is therefore the average of the distributions in the Delaunay mosaics and Voronoi tessellations; that is: $h = \frac{1}{2} (f+g)$.
  
  Finally, $\Iglesias{k}{A}$ is an orthogonal dual of $\Brillouin{k}{A}$, so its angles are supplementary to the angles in $\Brillouin{k}{A}$.
  Since $h$ is symmetric, $h(t) = h(\pi-t)$, this implies that $h$ is also the distribution of angles in $\Iglesias{k}{A}$.
\end{proof}

\subsection{Angle Distributions Experimentally}
\label{sec:4.3}

To get some perspective on Theorem~\ref{thm:monotonicity_of_extreme_angles}, we collect the angles in the Brillouin zones of $0$ in $\Zspace^2$ and in a perturbation of $\Zspace^2$.
Recall that the order-$k$ Brillouin tessellation of $\Zspace^2$ consists of a copy of the $k$-th Brillouin zone of $0$ for each point in the lattice, so we get the same angles either way.
In Figure~\ref{fig:angledistribution-IL}, we see the empirical distribution of the angles in the first $57$ Brillouin zones of $0 \in \Zspace^2$ in the left panel.
While the distribution for a generic point set is necessarily symmetric, we see a significant bias toward small angles.
The reason for the asymmetry are the many cocircular points, which give rise to high-degree vertices in the bisector arrangement.
If $\ell \geq 3$ lines pass through a common vertex, we get only the $2 \ell$ angles between lines that are contiguous in the ordering by slope.
The other $2 \binom{\ell}{2} - 2 \ell$ angles, which are necessarily larger, are suppressed by the degeneracy.
Compare this with the symmetric empirical distribution of angles on the right in Figure~\ref{fig:angledistribution-IL}, which collects the first $57$ Brillouin zones of $0$ in a perturbation of $\Zspace^2$.
\begin{figure}[hbt]
  \centering
    \vspace{0.1in}
    \includegraphics[width=0.49\textwidth,height=0.2\textwidth]{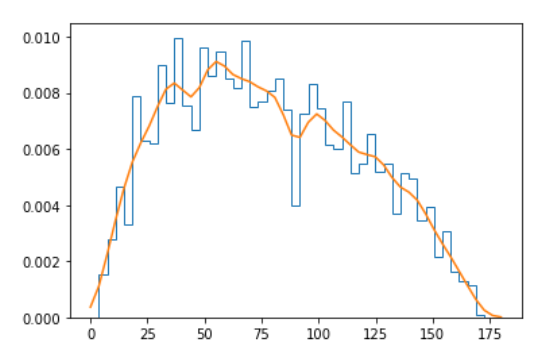} \hspace{0.00in}
    \includegraphics[width=0.47\textwidth,height=0.2\textwidth]{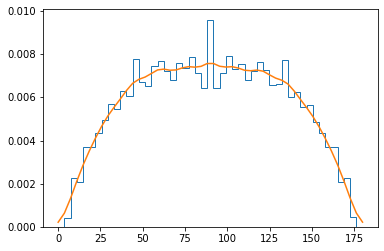}
    \vspace{-0.1in}
    \caption{The distribution of the angles in the first $57$ Brillouin zones of $0$ in $\Zspace^2$ on the \emph{left}, and in a perturbation of $\Zspace^2$ on the \emph{right}.
    The \emph{orange} curves smoothe out the histogram skylines.}
  \label{fig:angledistribution-IL}
\end{figure}

Figure~\ref{fig:PPP} shows empirical angle distributions for a stationary Poisson point process in $\Rspace^2$, which we approximate by sampling $400$ points uniformly at random in the unit square and copying this square around to avoid boundary effects.
According to Theorem~\ref{thm:angle_distributions}, we get the same distributions for each order, and these distributions are $f$ for the Delaunay mosaics, $g$ for the Voronoi tessellations, and $h$ for the Brillouin tessellations.
Indeed, the distributions for the four orders displayed in Figure~\ref{fig:PPP} have an unmistaken similarity to the distributions in Figure~\ref{fig:PPP-angles}.
\begin{figure}[hbt]
  \centering
    \vspace{0.1in}
    \includegraphics[width=0.48\textwidth,height=0.22\textwidth]{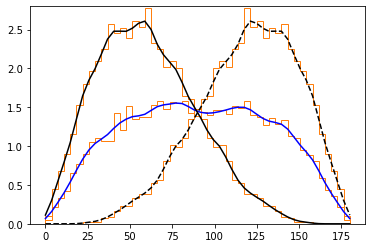} \hspace{0.1in}
    \includegraphics[width=0.48\textwidth,height=0.22\textwidth]{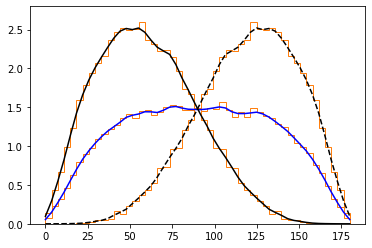} \\ \vspace{0.1in}
    \includegraphics[width=0.48\textwidth,height=0.22\textwidth]{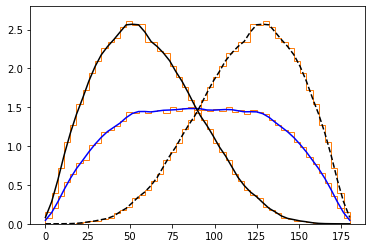} \hspace{0.1in}
    \includegraphics[width=0.48\textwidth,height=0.22\textwidth]{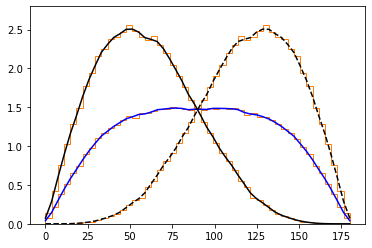}
    \vspace{-0.2in}
    \caption{The empirical angle distributions for the Delaunay mosaics, Voronoi tessellations, and Brillouin tessellations of a Poisson point process in $\Rspace^2$.
    \emph{Top:} orders $k = 2$ and $k = 6$.
    \emph{Bottom:} orders $k = 15$ and $k = 30$. }
  \label{fig:PPP}
\end{figure}

\section{Discussion}
\label{sec:5}

The main contribution of this paper is a theorem about the infimum angles in order-$k$ Brillouin tessellations and the supremum angles in their dual order-$k$ Iglesias mosaics in $\Rspace^2$.
In particular, these angles are monotonic in $k$ provided the point set that generates the tessellations and mosaics is locally finite and coarsely dense.
Without local finiteness, the concepts are not defined, and without coarse density, the angles fail to be monotonic.
Example sets that satisfy the requirements are lattices, periodic sets, Delone sets, and Poisson point processes, and they are used to illustrate the result with computational experiments.

If in addition to local finiteness and coarse density, we require that the point sets be generic, the infimum angles are also monotonic for the order-$k$ Delaunay and Iglesias mosaics, and the supremum angles are also monotonic for the order-$k$ Voronoi and Brillouin tessellations.
Without genericity, these angles fail to be monotonic.
We close this paper with two questions related to this work:
\medskip \begin{itemize}
  \item Is there a maxmin theorem for angles in order-$k$ Delaunay mosaics that extends Sibson's theorem \cite{Sib78} beyond $k = 1$?
  \item Is there an extension of Lawson's flip algorithm \cite{Law77} to Delaunay mosaics beyond order $1$?
\end{itemize}

\subsubsection*{Acknowledgements}

\small{The first author thanks Robert Adler and Anton Nikitenko for discussions on Poisson point processes, and for pointing out the formula for the angle distribution in \cite{Mil70}.}


\clearpage
\appendix

\clearpage
\section{Counterexample}
\label{app:A}

In this appendix, we prove that the omission of the Voronoi tessellations in the first family of inequalities in Theorem~\ref{thm:monotonicity_of_extreme_angles} is not accidental.
Equivalently, the omission of the Delaunay mosaics in the second family of inequalities in the same theorem is not accidental.
\begin{lemma}[Counterexample to Monotonicity]
  \label{lem:counterexample_to_monotonicity}
  For every sufficiently large $k$, there exist locally finite, coarsely dense, and generic sets $A \subseteq \Rspace^2$ such that $\alpha (\Voronoi{k}{A}) < \alpha (\Voronoi{k+1}{A})$ as well as $\omega (\Delaunay{k}{A}) > \omega (\Delaunay{k+1}{A})$.
\end{lemma}
\begin{proof}
  The two claimed inequalities are equivalent, so it suffices to prove the second.
  We do this by constructing a set that satisfies all claimed properties.
  Let $P \subseteq \Rspace^2$ be a perturbation of the integer lattice for which there exists a bijection $\gamma \colon \Zspace^2 \to P$ with $\Edist{a}{\gamma(a)} \leq \tau$, for some $0 < \tau < 1/4$.
  This perturbation is clearly locally finite and coarsely dense.
  It can also be chosen to be generic.
  
  \medskip
  Using a straightforward area argument, we see that a disk with radius $R$ contains at least $(R - \sfrac{\sqrt{2}}{2})^2 \pi$ and at most $(R + \sfrac{\sqrt{2}}{2})^2 \pi$ integer points.
  Similarly, it contains at least $(R - \sfrac{\sqrt{2}}{2} - \tau)^2 \pi$ and at most $(R + \sfrac{\sqrt{2}}{2} + \tau)^2 \pi$ points of $P$.
  It follows that the radius of a circle enclosing $k$ points of $P$ satisfies
  $
    \sqrt{{k}/{\pi}} - \sfrac{\sqrt{2}}{2} - \tau  \leq  R  \leq  \sqrt{{k}/{\pi}} + \sfrac{\sqrt{2}}{2} + \tau .
  $
  Letting $a, b, c \in P$ lie on such a circle,
  we have $\angle{acb} < \pi - {(1-2\tau)}/{R}$ because the arcs connecting $a$ to $c$ and $c$ to $b$ each have length greater than $1- 2 \tau$.
  The circumcircles of the triangles in the order-$(k+1)$ Delaunay mosaic enclose $k$ or $k-1$ point.
  We thus get an upper bound on the supremum angle from an upper bound on the possible radii:
  \begin{align}
    \omega (\Delaunay{k+1}{P})  &<  \pi - \frac{1-2\tau}{\sqrt{\frac{k}{\pi}} + \frac{\sqrt{2}}{2} + \tau} .
    \label{eqn:counter1}
  \end{align}
  Let $A$ be the set $P$ with two extra points: $a'$ and $a''$ at distance at most $\ee > 0$ from some point $a \in P$.
  Choose $a'$ and $a''$ generically so that the circle passing through $a', a, a''$ encloses $k-2$ points of $A$, and such that $a$ lies between $a'$ and $a''$ on the shorter of the two arcs connecting $a'$ to $a''$.
  For sufficiently small $\ee$, the lengths of the arcs from $a'$ to $a$ and from $a$ to $a''$ are less than $2 \ee$ each.
  Write $R'$ for the radius of this circle, and observe that in this case $\angle{a' a a''} > \pi - 2 \ee / R'$. 
  We get a lower bound on the supremum angle using a lower bound on the possible radii:
  \begin{align}
    \omega (\Delaunay{k}{A})  &>  \pi - \frac{2 \ee}{\sqrt{\frac{k-2}{\pi}} - \frac{\sqrt{2}}{2} - \tau} .
    \label{eqn:counter2}
  \end{align}
  For $k \geq 6$ the denominator is positive, which is the range of orders for which our construction works. 
  Most of the circles that enclose $k$ or $k-1$ points are the same for $A$ and for $P$.
  Since $P \subseteq A$, the upper bound on the possible radii used in \eqref{eqn:counter1} is still valid.
  Nevertheless, the upper bound on the supremum angle is somewhat smaller:
  \begin{align}
    \omega (\Delaunay{k+1}{A})  &<  \pi -  \frac{(1-2\tau-\ee)/2}{\sqrt{\frac{k}{\pi}} + \frac{\sqrt{2}}{2} + \tau} .
    \label{eqn:counter3}
  \end{align}
  because we may have circles that pass through two of $a', a, a''$ and some other third point.
  In this case, one arc can be very short while the other arc has length at least $1-2\tau-\ee$, which implies the bound in \eqref{eqn:counter3}.
  Choosing $\ee$ sufficiently small, we can make the lower bound in \eqref{eqn:counter2} arbitrarily close to $\pi$ while bounding the upper bound in \eqref{eqn:counter3} away from $\pi$.
  The claimed inequality follows.
\end{proof}

\end{document}